\documentclass{amsart}
\usepackage{amssymb}
\usepackage{hyperref}
\newcommand{\la}{\lambda}

\newcommand{\om}{\omega}

\newcommand{\be}{\begin{equation}}
	\newcommand{\ee}{\end{equation}}
\newcommand{\bea}{\begin{eqnarray}}
	\newcommand{\eea}{\end{eqnarray}}

\newcommand{\bee}{\begin{eqnarray*}}
	\newcommand{\eee}{\end{eqnarray*}}
\newcommand{\ba}{\begin{aligned}}
	\newcommand{\ea}{\end{aligned}}
\newcommand{\bp}{\begin{proof}}
	\newcommand{\ep}{\end{proof}}  
\newcommand{\br}{\begin{remark}}
	\newcommand{\er}{\end{remark}}  
\newcommand{\lb}{\label}

\newtheorem{thm}{Theorem}[section]

\newtheorem{lem}[thm]{Lemma}

\theoremstyle{definition}
\newtheorem{defn}[thm]{Definition}

\theoremstyle{remark}

\usepackage{enumitem}

\newcommand{\Field}{\mathbb{F}}

\begin{document}
	\title{Existence of primitive normal pairs over finite fields with prescribed subtrace}
	\author[K. Chatterjee]{Kaustav Chatterjee$^*$}
	\address{Department of Mathematics,  Indian Institute of Technology Patna, BIHAR, INDIA.}
	\email{kaustav0004@gmail.com}
	
	\author[G. Kapetanakis]{Giorgos Kapetanakis}
	\address{Department of Mathematics, University of Thessaly, 3rd km Old National Road 
		Lamia-Athens, 35100, Lamia, Greece}
	\email{kapetanakis@uth.gr}

	\author[H. Sharma]{Hariom Sharma}
	\address{ Department of Mathematics, S. S. Govt. P.G. College Tigaon, HARYANA, INDIA.}
	\email{hariomsharma638@gmail.com}
	
	\author[S.K. Tiwari]{Shailesh Kumar Tiwari}
	\address{ Department of Mathematics,  Indian Institute of Technology Patna, BIHAR, INDIA.}
	\email{sktiwari@iitp.ac.in, shaileshiitd84@gmail.com}

	%\thanks{ *email:sktiwari@iitp.ac.in}
	\thanks{*email:kaustav0004@gmail.com}
	
	%\vspace*{.6cm}
	\begin{abstract}
		Given positive integers $q,n,m$ and $a\in\Field_{q}$, where $q$ is an odd prime power and $n\geq 5$, we investigate the existence of a primitive normal pair $(\epsilon,f(\epsilon))$ in $\Field_{q^{n}}$ over $\Field_{q}$ such that $\mathrm{STr}_{q^n/q}(\epsilon)=a$, where $f(x)=\frac{f_{1}(x)}{f_{2}(x)}\in\mathbb{F}_{q^n}(x)$ is a rational function together with deg$(f_{1})+$deg$(f_{2})=m$ and $\mathrm{STr}_{q^n/q}(\epsilon) = \sum_{0\leq i<j\leq n-1}^{}\epsilon^{q^i+q^j}$. Finally, we conclude that for $m=2$, $n\geq 6$ and $q=7^k$; $k\in\mathbb{N}$,  such a pair will exist certainly for all $(q,n)$ except at most $11$ choices.
	\end{abstract}
	\maketitle
	\section{Introduction}
	Assume that $\Field_{q}$ represents a finite field containing $q$ elements, where $q$ is of the form $p^k$, where $k$ is a positive integer and $p$ is a prime. The set $\mathbb{F}_q^*=\Field_{q}\smallsetminus \{0\}$ forms a cyclic group with respect to multiplication and any generator of the multiplicative group $\mathbb{F}_q^*$ is defined to be a primitive element of $\Field_{q}$. We always get $\phi(q-1)$ number of primitive elements in the finite field $\Field_{q}$, where $\phi$ represents the Euler's totient function. For some primitive $\epsilon\in\Field_{q^{n}}^*$, if there exists a non constant polynomial of least degree over $\Field_{q}$ such that $\epsilon$ is one of it's roots, then the corresponding polynomial is defined as primitive polynomial of $\epsilon$. Let $\Field_{q^n}$ denotes a finite extension of the finite field $\Field_{q}$ of degree of $n$, where $n$ represents a natural number. The elements $\epsilon, \epsilon^{q}, \ldots, \epsilon^{q^{n-1}}$ are said to be the conjugates of $\epsilon$ with respect to $\Field_{q}$, for any $\epsilon$ $\in\Field_{q^n}$. In particular, if for some $\epsilon\in\Field_{q^{n}}$, the set contained with the conjugates of $\epsilon$ creates a basis of $\Field_{q^n}$ over $\Field_q$, then $\epsilon$ is said to be normal, whereas the basis itself is referred as a normal basis. A pair $(\epsilon,\delta) \in\Field_{q^n}^*\times \Field_{q^n}^*$ is defined to be a primitive normal pair provided both $\epsilon$ and $\delta$ are primitive along with normal over the finite field $\Field_{q}$. In order to get additional details regarding the  primitive normal elements in finite fields, we prefer \cite{RH} to the reader. In this article, let us introduce the following definition.
	
	For any $\epsilon\in\Field_{q^{n}}$, the \emph{subtrace} of $\epsilon$ over $\Field_{q}$ is denoted by $\mathrm{STr}_{q^n/q}(\epsilon)$ and is defined by 
	\begin{equation}\nonumber
		\mathrm{STr}_{q^n/q}(\epsilon)=\sum_{0\leq i<j\leq n-1}^{}\epsilon^{q^i+q^j}.
	\end{equation}
	In other words, considering the products $\epsilon^{q^i}\cdot\epsilon^{q^j} (0\leq i<j \leq n-1)$ of any two distinct conjugates and then summing up these products, we get the subtrace of $\epsilon$ over $\Field_{q}$. We now represent the subtrace in terms of minimum polynomial as follows.  Assume that $\Psi(x)=x^n+a_{n-1}x^{n-1}+a_{n-2}x^{n-2}+\ldots+a_{0}\in\Field_{q}[x]$ be the minimum polynomial of $\epsilon$.  Again, the roots of $\Psi$ are $\epsilon, \epsilon^{q}, \ldots, \epsilon^{q^{n-1}}$. Thus 
	\begin{equation}\nonumber
		\begin{aligned}
			\Psi(x)&=x^n+a_{n-1}x^{n-1}+a_{n-2}x^{n-2}+\ldots+a_{0}\\
			&=(x-\epsilon)(x-\epsilon^q)(x-\epsilon^{q^2})\ldots(x-\epsilon^{q^{n-1}}),
		\end{aligned}
	\end{equation}
	and equating the coefficients we have $\mathrm{STr}_{q^n/q}(\epsilon)=a_{n-2}$, that is, $\mathrm{STr}_{q^n/q}(\epsilon)\in\Field_{q}$.

	Previous to the paper, sufficient conditions were established for the existence of a primitive normal pair $(\epsilon,f(\epsilon))$ together with prescribed trace or norm, where $f(x)\in\Field_{q^{n}}(x)$ rational function having some minor restrictions . In this article, we try to find those pairs $(q,n)$ such that the field $\Field_{q^{n}}$ contains at least a primitive normal pair $(\epsilon,f(\epsilon))$ over $\Field_{q}$ such that $\mathrm{STr}_{q^n/q}(\epsilon)=a$ for prescribed $a\in\Field_{q}$. Further, the \emph{trace} of an element $\epsilon \in \Field_{q^n}$ over $\Field_{q}$, denoted by $\mathrm{Tr}_{{q^n}/q}(\epsilon)$ and is defined by $\mathrm{Tr}_{{q^n}/q}(\epsilon)=\epsilon+\epsilon^{q}+\ldots+\epsilon^{q^{n-1}}$. To proceed further, we shall utilize the following result.
	\begin{lem}\lb{L1.2}
		Assume that $q$, $n$ be positive integers where $q$ is an odd prime power and $a\in\Field_{q}$. Then
		$\mathrm{STr}_{q^n/q}(\epsilon)=\{{\mathrm{Tr}_{q^n/q}(\epsilon)}^2-\mathrm{Tr}_{q^n/q}(\epsilon^2)\}\cdot2^{-1}$ for any $\epsilon\in\Field_{q^{n}}$. Moreover, if  ${\mathrm{Tr}_{q^n/q}(\epsilon)}^2=a$ and $\mathrm{Tr}_{q^n/q}(\epsilon^2)=-a$, then $\mathrm{STr}_{q^n/q}(\epsilon)=a$.
	\end{lem}
	\bp
	For any $\epsilon\in\Field_{q^{n}}$, 
	\begin{equation}\nonumber
		\begin{aligned}
			{\mathrm{Tr}_{q^n/q}(\epsilon)}^2-\mathrm{Tr}_{q^n/q}(\epsilon^2)&=\Bigg(\sum_{i=0}^{n-1}\epsilon^{q^i}\Bigg)^2-\sum_{j=0}^{n-1}\epsilon^{2q^j}\\
			&=\sum_{i=0}^{n-1}\epsilon^{2q^i}+2\cdot\sum_{0\leq i<j\leq n-1}^{}\epsilon^{q^i+q^j}-\sum_{j=0}^{n-1}\epsilon^{2q^j}\\
			&=2\cdot\mathrm{STr}_{q^n/q}(\epsilon).
		\end{aligned}
	\end{equation}
	Since $q$ is odd prime power, we have $\mathrm{STr}_{q^n/q}(\epsilon)=\{{\mathrm{Tr}_{q^n/q}(\epsilon)}^2-\mathrm{Tr}_{q^n/q}(\epsilon^2)\}\cdot2^{-1}$. Thus we have
	\begin{equation}\nonumber
		\begin{aligned}
			\mathrm{STr}_{q^n/q}(\epsilon)-a&=\{{\mathrm{Tr}_{q^n/q}(\epsilon)}^2-\mathrm{Tr}_{q^n/q}(\epsilon^2)\}\cdot2^{-1}-a\\
			&=\{({\mathrm{Tr}_{q^n/q}(\epsilon)}^2-a)-(\mathrm{Tr}_{q^n/q}(\epsilon^2)+a)\}\cdot 2^{-1},
		\end{aligned}
	\end{equation}
	which gives the desired result. 
	\ep 
	Let $a,b\in\Field_{q}$ be such that $a$ is the square of $b$. Referring the above lemma, we simplify our problem to the existence of a primitive normal pair $(\epsilon,f(\epsilon))$, where $f(x)\in\mathbb{F}_{q^n}(x)$ is a rational function of degree sum $m$ together with $\mathrm{Tr}_{q^n/q}(\epsilon)=b$ and $\mathrm{Tr}_{q^n/q}(\epsilon^2)=-a$. Throughout this article, $q$ is denoted to be any odd prime power. In fact, for $f(x)\in\Field_{q^{n}}(x)$, investigating the existence of primitive normal pair $(\epsilon,f(\epsilon))$ together with prescribed trace has been a much interesting research area. Numerous scholars have contributed to this domain, as documented in works such as \cite{WS, CW, HR, HRK, ARS, AR}.

	For $m_{1},m_{2}\in\mathbb{N}\cup \{0\}$, we define the following sets, which will play an important role throughout this article. 
	\begin{enumerate}
		\item[1.] Let us define $\mathcal{R}_{q,n}(m_{1},m_{2})$ to be the set contains with the rational functions $f(x)=\frac{f_{1}(x)}{f_{2}(x)}$, where $f_{1}$ and $f_{2}$ are irreducible polynomials over $\Field_{q^n}$ (or nonzero constants in $\Field_{q^n}$) such that gcd$(f_{1},f_{2})=1$ with $deg(f_{1})=m_{1}$ and $deg(f_{2})=m_{2}$.\\
		\item[2.] Let $\mathcal{D}_{m_{1},m_{2}}$ be the set which contains the pairs $(q,n)\in\mathbb{N}\times\mathbb{N}$ such that for any $f\in\mathcal{R}_{q,n}(m_{1},m_{2})$ and prescribed $a\in\mathbb{F}_{q}$, we get a primitive normal pair $(\epsilon,f(\epsilon))\in\Field_{q^{n}}^*\times\Field_{q^{n}}^*$ for which $\mathrm{Tr}_{q^n/q}(\epsilon)=b$ and $\mathrm{Tr}_{q^n/q}(\epsilon^2)=-a$, where $b\in\Field_{q}$ be such that $a=b^2$.\\
		\item[3.] Define , $\mathcal{R}_{q,n}(m)=\bigcup_{m_{1}+m_{2}=m}\mathcal{R}_{q,n}(m_{1},m_{2})$ and $\mathcal{D}_{m}=\bigcap_{m_{1}+m_{2}=m}\mathcal{D}_{m_{1},m_{2}}$.
	\end{enumerate}
	Clearly, $(q,1)\notin \mathcal{D}_{m_{1},m_{2}}$ as in that case we get that $\mathrm{STr}_{q^n/q}(\epsilon)=\epsilon$. Hence $(q,1)$ to be in $\mathcal{D}_{m_{1},m_{2}}$, every pair $(\epsilon,f(\epsilon))$ in $\Field_{q}$ must be primitive normal,  for any $f\in\mathcal{R}_{q,n}(m_{1},m_{2})$, which is possible only if $q-1$ is prime, that is, if $p=2$. We assume that $f(x)=x+1$. Then it implies that $(1,0)$ is a primitive normal pair. However if $n=2$, then $\mathrm{STr}_{q^n/q}(\epsilon)=\epsilon^{q+1}$ and so there does not exist any primitive normal element with subtrace $1$. Thus $(q,2)\notin \mathcal{D}_{m_{1},m_{2}}$. For $n=3$, let us take $f(x)=1/x$ and thus the problem of prescribing $0$ as subtrace is equivalent to the existence of a primitive normal pair $(\epsilon,\epsilon^{-1})$ together with ${Tr}_{q^n/q}(\epsilon^{-1})=0$, an impossibility. Hence $(q,3)\notin \mathcal{D}_{m_{1},m_{2}}$. Finally, when $n=4$, by letting $f(x)=(x-1)$ and intending to find a primitive normal pair $(\epsilon,f(\epsilon))$ together with $\mathrm{Tr}_{q^n/q}(\epsilon)=4$ and $\mathrm{Tr}_{q^n/q}(\epsilon^2)=-4^2$, that is, $\mathrm{Tr}_{q^n/q}(\epsilon-1)=0$ and $\mathrm{Tr}_{q^n/q}(\epsilon^2)=-4^2$, we get a contradiction. This gives $(q,4)\notin \mathcal{D}_{m_{1},m_{2}}$. Therefore, throughout this article, we shall consider $n\geq 5$.

	This article is outlined as follows.  Important notations and definitions to be utilized right through the article are mentioned in  Section $\ref{S2}$. Section $\ref{S3}$ contains the sufficient condition for attaining our goal. Next, in Section $\ref{S4}$, we give prime sieve condition, which relaxes the sufficient condition, followed by a modified version of prime sieve. Finally, we illustrate the application of the results of in the previous sections by taking the finite fields with characteristic $7$ and $m=2$. To be more specific, we obtain a subset of $\mathcal{D}_2$.  
	\section{Preliminaries}\lb{S2}
	This section presents an introductory overview of fundamental ideas, symbols, and definitions that will be utilized all over this article. Here, $n$ represents a positive integer, $q$ is used to denote any odd prime power and $\mathbb{F}_{q}$ is referred to be the finite field containing $q$ number of elements.
	\begin{defn}[Character]
		Let $A$ be an abelain group and $U$ be the subset of complex numbers containing elements on the circle with unit modulus. A character $\chi$ of $A$ is a homomorphism from $A$ into $U$, i.e., $\chi(a_{1}a_{2})=\chi(a_{1})\chi(a_{2})$ for all $a_{1},a_{2}\in A$. 
	\end{defn}
	The character $\chi_{1}$ defined by $\chi_{1}(a)=1$ for all $a\in A$, is said to be the trivial character of $A$. Moreover, the collection of all characters of $A$, denoted as $\widehat{A}$, forms a group under multiplication and $A\cong \widehat{A}$. Further, since $\Field_{q^{n}}^*\cong \widehat{\Field_{q^{n}}^*}$, then for any $d|q^n-1$, there are $\phi(d)$ multiplicative characters of order $d$.
	
	In the context of a finite field $\Field_{q^{n}}$, a multiplicative character relates to the multiplicative group $\Field_{q^{n}}^*$, whereas an additive character corresponds to the additive group $\Field_{q^{n}}$. Any multiplicative character can $\chi$, associated to $\Field_{q^{n}}^*$, can be extended to $\Field_{q^{n}}$ by the following rule
	\[
	\chi(0)=\begin{cases}
		
		1, ~\text{if} ~\chi= \chi_{1},  \\
		0, ~\text{if} ~\chi\neq \chi_{1}.
	\end{cases}\]
	\begin{defn}[$l$-free element]
		Let $\epsilon \in \Field_{q^{n}}^{*}$ and $l$ be any divisor of $q^n-1$. Then, $\epsilon$ is said to be a $l$-free, if $\epsilon=\delta^{j}$, where $\delta \in\Field_{q^{n}}^*$ and $j|l$ implies that $j=1$. Clearly, $\epsilon\in\Field_{q^{n}}^*$ is primitive if and only if $\epsilon$ is $(q^{n}-1)$-free.
	\end{defn}
	For any $l|(q^n-1)$, the characteristic function determining $l$-free elements of $\Field_{q^n}^*$ is given by 
	\begin{equation}\lb{E1}
		\begin{aligned}
			\rho_{l}:\Field_{q^n}^*\rightarrow \{0,1\};\epsilon \mapsto{}\theta(l)\sum_{d|l}\frac{\mu(d)}{\phi(d)}\sum_{\chi_{d}}\chi_{d}(\epsilon),
		\end{aligned}
	\end{equation}
	where $\theta(l):=\frac{\phi(l)}{l}$, $\chi_{d}$ represents a multiplicative character of order $d$ in $\widehat{\Field_{q^n}^*}$ and $\mu$ is the M\"{o}bius function.

	The additive group $\Field_{q^{n}}$ becomes an $\Field_{q}[x]$-module according to the following rule.
	$$f\circ\epsilon=\sum_{i=0}^{r}{a_{i}\epsilon^{q^{i}}}; ~\text{for}~ \epsilon\in\Field_{q^{n}}~ \text{and}~ f(x)=\sum_{i=0}^{r}{a_{i}x^{i}}\in \Field_{q}[x].$$ For $\epsilon\in\Field_{q^{n}}$, the $\Field_{q}$-order of $\epsilon$ is the monic $\Field_{q}$-divisor $f$ of $x^n-1$ of minimal degree such
	that $f\circ \epsilon= 0$.
	\begin{defn}[$g$-free element]
		Let $g|x^n-1$ and $\epsilon\in\Field_{q^{n}}$. Suppose that for any $h|g$ and $\delta\in\Field_{q^{n}}$, $\epsilon=h\circ\delta$ implies $h=1$. Then $\epsilon\in\Field_{q^{n}}$ is said to be $g$-free. It is straightforward to observe that, any element $\epsilon\in\Field_{q^{n}}$ is normal if and only if $\epsilon$ is $(x^n-1)$-free.
	\end{defn}
	For any $g|x^n-1$, the characteristic function determining the subset of $g$-free elements in $\Field_{q^{n}}$ is given as follows
	\begin{equation}\lb{E2}
		\begin{aligned}
			\kappa_{g}:\Field_{q^n}\mapsto \{0,1\};\epsilon\xrightarrow{}\Theta(g)\sum_{h|g}\frac{\mu_{q}(h)}{\Phi_{q}(h)}\sum_{\la_{h}}\la_{h}(\epsilon),
		\end{aligned}
	\end{equation}
	where $\Theta(g):=\frac{\Phi_{q}(g)}{q^{deg(g)}}$, $\la_{h}$ stands for any additive character of $\Field_{q}$-order $h$ in $\widehat{\Field_{q^n}}$ and $\mu_{q}$ is the M\"{o}bius function for the set of polynomials over $\Field_{q}$ is defined as follows:
	\[
	\mu_{q}(f)=\begin{cases}
		(-1)^{r}, & \text{if $f$ is product of $r$ distinct monic irreducible polynomials},\\
		0, & \text{otherwise}.
	\end{cases}
	\]
	
	For any $a\in\Field_{q}$, the characteristic function for the subset of $\Field_{q^n}$ containing elements with $\mathrm{Tr}_{{q^n}/q}(\epsilon)=a$ is given as follows:
	\begin{equation}\nonumber
		\tau_{a}:\Field_{q^n}\rightarrow \{0,1\};\epsilon \mapsto \frac{1}{q}\sum_{\la\in{\widehat{{\Field}_{q}}}}^{}\la(\mathrm{Tr}_{{q^n}/q}(\epsilon)-a).
	\end{equation} 
	Now, any additive character $\la$ of $\Field_{q}$ can be represented using the canonical additive character $\la_{0}$ as $\la(\epsilon)=\la_{0}(t\epsilon)$, where $t$ is an element of $\Field_{q}$ that corresponds to $\la$. Thus
	\begin{equation}\lb{E3}
		\begin{aligned}
			\tau_{a}(\epsilon)&= \frac{1}{q}\sum_{t\in\Field_{q}}^{}\la_{0}(\mathrm{Tr}_{{q^n}/q}(t\epsilon)-ta)\\
			&=\frac{1}{q}\sum_{t\in\Field_{q}}^{}{\widehat{\la}_{0}}(t\epsilon)\la_{0}(-ta),
		\end{aligned}
	\end{equation}
	where $\widehat{\la}_{0}$ is the additive character of $\Field_{q^n}$, that is given by $\widehat{\la}_{0}(\epsilon)=\la_{0}(\mathrm{Tr}_{{q^n}/q}(\epsilon))$.

	The following results are given by Weil \cite{AW1948} as described in \cite{TC2006} at $(1.1)$ to $(1.3)$.
	\begin{lem}\lb{L2.1}
		Consider $f(x)=\prod_{i=1}^{r}f_{i}(x)^{a_{i}}\in\mathbb{F}_{q^n}(x)$ be such that $f_{i}$'s are irreducible polynomials over $\Field_{q^{n}}$ and $a_{i}$'s are nonzero integers. Suppose that $\chi\in\widehat{{\Field}_{q^n}^*}$ be a multiplicative character having order $d$. Further, assume $f(x)$ to be a rational function, which is not equal to $ch(x)^{d}$, for $c\in\Field_{q^n}^*$ and $h(x)\in\mathbb{F}_{q^n}(x)$. Then  $$\Bigg|\sum_{\epsilon\in\Field_{q^n}, f(\epsilon)\neq \infty }^{}\chi(f(\epsilon))\Bigg|\leq\Bigg(\sum_{i=1}^{r}deg(f_{i})-1\Bigg)q^{n/2}.$$	
	\end{lem}
	\begin{lem}[{\cite{CM2000}}]\lb{L2.2}
		Let $\chi$ be non-trivial multiplicative of order $r$ and $\la$ be a non-trivial additive character on $\Field_{q^n}$. Let $f$ and $g$ be the rational functions in $\Field_{q^n}(x)$ such that  $f\neq yh^r$, for any $y\in\Field_{q^n}$, and $g\neq h^p-h+\beta$, where $h\in\Field_{q^n}(x)$ and $\beta\in\Field_{q^n}$.   Then, we have
		\begin{equation} \nonumber
			\begin{aligned}
				\Bigg|\sum_{\epsilon\in \Field_{q^{n}}\smallsetminus\mathcal{S}}\chi(f(\epsilon))\la(g(\epsilon))\Bigg|\leq(deg(g_\infty)+l+l'-l''-2)q^{n/2},
			\end{aligned}
		\end{equation}
		where $\mathcal{S}$ is set the poles of $f$ and $g$, $(g)_\infty$ is the pole divisor of $g$,  $l$ is the number of distinct zeros and (non-infinite) poles of $f$, $l'$ is the number of distinct poles of $g$ (including $\infty$) and $l''$ is the number of finite poles of $f$ that are poles or zeros of $g$.
	\end{lem}
	\begin{lem}\textbf{({\cite{RH}}, Theorem 5.38)}\lb{L2.6}
		Let $f\in\Field_{q}[x]$ of degree $n\geq 1$ with gcd$(n,q)=1$ and $\la$ be a non-trivial additive character of $\Field_{q}$. Then 
		$$\Bigg|\sum_{\alpha\in\Field_{q}}^{}\la(f(\epsilon))\Bigg|\leq(n-1)q^{1/2}.$$
	\end{lem}
	For $l$, a positive integer (or a monic polynomial over $\Field_{q}$), we use $\om(l)$ to represent  the cardinality of the set which contains distinct prime divisors (irreducible factors ) of $l$ and $W(l)$ to represent the cardinality of the set which contains square-free divisors (square-free factors) of $l$, that is $W(l)=2^{\om(l)}$. 
	\begin{lem}[\cite{SS}, Lemma 3.7]\lb{L2.3}
		Let $r>0$ be a real number and $m$ be a positive integer. Then $W(m)<\mathcal{C}\cdot m^{\frac{1}{r}}$, where $\mathcal{C}=\frac{2^w}{{({p_{1}p_{2}\ldots p_{w})}}^{\frac{1}{r}}}$ and $p_{1},p_{2},\ldots,p_{w}$ are primes $\leq 2^{r}$ that divide $m$.
	\end{lem}
	\begin{lem}[\cite{HWRJ}, Lemma 2.9]\lb{L2.4}
		Suppose that $q$ be a prime power, $n$ be a natural number and $n'$=gcd$(n,q-1)$. Then we have $W(x^n-1)\leq 2^{\frac{1}{2}\{n+n'\}}$, which gives $W(x^n-1)\leq 2^n$. Further, $W(x^n-1)=2^n$ if and only if $n|q-1$. In addition, if  $n\nmid q-1$, then $W(x^n-1)\leq 2^{\frac{3}{4}n}$.
	\end{lem}
	\section{Main Result}\lb{S3}
	Let $q,n,m\in\mathbb{N}$ such that $q$ is an odd prime power, $n\geq 5$. Let $l_{1},l_{2}|q^n-1$ and $g_{1},g_{2}|x^n-1$. Also, let $f(x)\in\mathcal{R}_{q,n}(m)$ and $a,b\in\Field_{q}^*$ be such that $a$ is the square of $b$. Thus, to show that $\epsilon\in\Field_{q^{n}}^*$ is $g_1$-free, it is equivalent to prove that $\epsilon$ is $L_{g_1}$-free, where $L_{g_1}$ is the largest divisor of $g_1$ such that $(x-1)\nmid L_{g_1}$. We consider $\mathfrak{C}_{f,a,b}(l_{1},l_{2},L_{g_1},g_{2})$ as the number of $\epsilon\in\Field_{q^{n}}^*$ such that $\epsilon$ is $l_{1}$-free, $L_{g_1}$-free and $f(\epsilon)$ is $l_{2}$-free, $g_{2}$-free together with $\mathrm{Tr}_{q^n/q}(\epsilon)=b$ and $\mathrm{Tr}_{q^n/q}(\epsilon^2)=-a$. Set $L:=L_{x^n-1}$.

	\begin{thm}\lb{T3.1}
		Let $n,m,q\in\mathbb{N}$ be such that $q$ is an odd prime power and $n\geq 5$. Assume that
		$$q^{\frac{n}{2}-2}>2m{W(q^n-1)}^2W(L)W(x^n-1).$$
		Then $(q,n)\in\mathcal{D}_{m}$.
	\end{thm}
	\bp
	With the purpose of proving the above result, initially we demonstrate that $\mathfrak{C}_{f,a,b}(l_{1},l_{2},L_{g_1},g_{2})>0$ for any $f\in\mathcal{R}_{q,n}(m)$ and for prescribed $a,b\in\Field_{q}^*$, where $a$ is the square of $b$. Let us consider the sets $\mathcal{Z}=\{\epsilon\in\Field_{q^{n}}:f_{1}(\epsilon)=0, f_{2}(\epsilon)=0\}$ and $\mathcal{Z}_{1}=\mathcal{Z}\cup \{0\}$. Then by the definition, $\mathfrak{C}_{f,a,b}(l_{1},l_{2},L_{g_1},g_{2})$ is expressed as follows.
	\begin{equation}\nonumber
		\begin{aligned}
			\sum_{\epsilon\in\Field_{q^{n}}\smallsetminus \mathcal{Z}_{1}}^{}\rho_{l_{1}}(\epsilon)\rho_{l_{2}}(f(\epsilon))\kappa_{L_{g_1}}(\epsilon)\kappa_{g_{2}}(f(\epsilon))\tau_{b}(\epsilon)\tau_{-a}(\epsilon^2).
		\end{aligned}
	\end{equation}
	Using (\ref{E1}), (\ref{E2}) and (\ref{E3}), we have
	\begin{equation}\lb{E4}
		\begin{aligned}
			\mathfrak{C}_{f,a,b}(l_{1},l_{2},L_{g_1},g_{2})=\mathcal{G}\underset{\underset{h_{1}|L_{g_1},h_{2}|g_{2}}{d_{1}|l_{1}, d_{2}|l_{2}}}{\sum}\frac{\mu}{\phi}(d_{1},d_{2},h_{1},h_{2})\underset{\underset{\la_{h_{1},\la_{h_{2}}}}{\chi_{d_{1}},\chi_{d_{2}}}}{\sum}\boldsymbol{\chi}_{f,a,b}(d_{1},d_{2},h_{1},h_{2}) 
		\end{aligned}
	\end{equation}
	where $\mathcal{G}=\frac{\theta(l_{1})\theta(l_{2})\Theta(L_{g_1})\Theta(g_{2})}{q^2}$, $\frac{\mu}{\phi}(d_{1},d_{2},h_{1},h_{2})=\frac{\mu(d_{1})\mu(d_{2})\mu_{q}(h_{1})\mu_{q}(h_{2})}{\phi(d_{1})\phi(d_{2})\Phi_{q}(h_{1})\Phi_{q}(h_{2})}$ and
	\begin{equation}\nonumber
		\begin{aligned}
			\boldsymbol{\chi}_{f,a,b}(d_{1},d_{2},h_{1},h_{2}) &=\sum_{t,u\in\Field_{q}}^{}\la_{0}(-bt+au)\sum_{\epsilon\in\Field_{q^{n}}\smallsetminus \mathcal{Z}_{1}}^{}\chi_{d_{1}}(\epsilon)\chi_{d_{2}}(f(\epsilon))\la_{h_{1}}(\epsilon)\la_{h_{2}}(f(\epsilon))\\&\times \widehat{\la_{0}}(t\epsilon+u\epsilon^2).
		\end{aligned}
	\end{equation}
	Note that, for $d_{1}, d_{2} |q^n-1$, there exist integers $c_{1}$, $c_{2}$ with $0\leq c_{1},c_{2} \leq q^n-2$ such that $\chi_{d_{1}}(\epsilon)=\chi_{q^n-1}(\epsilon^{c_{1}})$ and $\chi_{d_{2}}(\epsilon)=\chi_{q^n-1}(\epsilon^{c_{2}})$. 
	Further, for $h_{1}, h_{2} |x^n-1$, there always exist $y_{1},y_{2}\in\Field_{q^{n}}$ so that $\la_{h_{1}}(\epsilon)=\widehat{\la_{0}}(y_{1}\epsilon)$ and $\la_{h_{2}}(\epsilon)=\widehat{\la_{0}}(y_{2}\epsilon)$. Hence we get
	\begin{equation}\nonumber
		\begin{aligned}
			\boldsymbol{\chi}_{f,a,b}(d_{1},d_{2},h_{1},h_{2}) =&\sum_{t,u\in\Field_{q}}^{}\la_{0}(-bt+au)\sum_{\epsilon\in\Field_{q^{n}}\smallsetminus \mathcal{Z}_{1}}^{}\chi_{q^n-1}(\epsilon^{c_{1}}f(\epsilon)^{c_{2}})\times\\&\widehat{\la_{0}}((t+y_{1})\epsilon+u\epsilon^2+y_{2}f(\epsilon))\\
			=&\sum_{t,u\in\Field_{q}}^{}\la_{0}(-bt+au)\sum_{\epsilon\in\Field_{q^{n}}\smallsetminus \mathcal{Z}_{1}}^{}\chi_{q^n-1}(R_{1}(\epsilon))\widehat{\la_{0}}(R_{2}(\epsilon)),
		\end{aligned}
	\end{equation}
	where $R_{1}(x)=x^{c_{1}}f(x)^{c_{2}}\in\Field_{q^{n}}(x)$ and $R_{2}(x)=(t+y_{1})x+ux^2+y_{2}f(x)\in\Field_{q^{n}}(x)$.

	First, we consider the situation when $R_{2}(x)={r(x)}^{p}-r(x)+\beta$ for some $r(x)\in\Field_{q^{n}}(x)$ and $\beta\in\Field_{q^n}$, where $r(x)=\frac{r_{1}(x)}{r_{2}(x)}$ with $\gcd(r_{1},r_{2})=1$. Thus we get 
	\[ (t+y_{1})x+ux^2+y_{2}\frac{f_{1}(x)}{f_{2}(x)}=\frac{{r_{1}(x)}^{p}}{{r_{2}(x)}^{p}}-\frac{r_{1}(x)}{r_{2}(x)}+\beta, \]
	that is,
	\[ f_{2}(x)({r_{1}(x)}^{p}-r_{1}(x){r_{2}(x)}^{p-1}+\beta r_{2}(x)^p)={r_{2}(x)}^{p}((t+y_{1})xf_{2}(x)+ux^2f_{2}(x)+y_{2}f_{1}(x)). \]
	As $\gcd \left({r_{1}(x)}^{p}-r_{1}(x){r_{2}(x)}^{p-1}+\beta{r_{2}(x)}^{p},{r_{2}(x)}^{p}\right)=1$, it gives that ${r_{2}(x)}^{p}|f_{2}(x)$, which is not possible until we have $r_{2}=c$, a nonzero constant. Consequently, we have 
	$${c}^{-p}f_{2}(x)({r_{1}(x)}^{p}-r_{1}(x){c}^{p-1}+\beta c^p)=(t+y_{1})xf_{2}(x)+ux^2f_{2}(x)+y_{2}f_{1}(x),$$
	which holds true only if $y_{2}=0$.  Putting it into the above equation yields, ${c}^{-p}({r_{1}(x)}^{p}\\-r_{1}(x)c^{p-1}+\beta c^p)=(t+y_{1})x+ux^2$, which is possible only if $t+y_{1}=0$, $u=0$. However, note that $\la_{h_{1}}(h_{1}(\epsilon))=1$, which further implies $\widehat{\la_{0}}(y_{1}h_{1}(\epsilon))=1$, that is, $\widehat{\la_{0}}(-th_{1}(\epsilon))=1$. Since $t\in\Field_{q}$ and $x-1$ is the $\Field_{q}$-order of $\widehat{\la_{0}}$, therefore we must have $x-1|th_{1}$, which is true only if $t=0$, and this implies $y_{1}=0$. Hence $y_{1}=y_{2}=0$, that is, $h_{1}=h_{2}=1$. In addition, if $R_{1}(x)\neq c{l(x)}^{q^n-1}$ for any $l(x)\in\Field_{q^{n}}(x)$ and $c\in\Field_{q^n}^*$, then from Lemma \ref{L2.1}, it follows that
	\begin{equation}\nonumber
		\begin{aligned}
			|\boldsymbol{\chi}_{f,a,b}(d_{1},d_{2},1,1) |\leq mq^{\frac{n}{2}+2}.
		\end{aligned}
	\end{equation}
	We now assume that $R_{1}(x)=cl(x)^{q^n-1}$ for some $c\in\Field_{q^{n}}^*$ and $l(x)\in\Field_{q^{n}}(x)$, where $l(x)=\frac{l_{1}(x)}{l_{2}(x)}$ with $\gcd(l_{1},l_{2})=1$. Following \cite{HR}, it is true only when $c_{1}=c_{2}=0$, that is, $d_{1}=d_{2}=1$. Here we have \\
	\begin{equation}\nonumber
		\begin{aligned}
			\boldsymbol{\chi}_{f,a,b}(1,1,1,1)&=\sum_{t,u\in\Field_{q}}^{}\la_{0}(-bt+au)\sum_{\epsilon\in\Field_{q^{n}}\smallsetminus \mathcal{Z}_{1}}^{} \widehat{\la_{0}}(t\epsilon+u\epsilon^2)\\&=\sum_{t,u\in\Field_{q}}^{}\la_{0}(-bt+au)\sum_{\epsilon\in\Field_{q^{n}}}^{} \widehat{\la_{0}}(t\epsilon+u\epsilon^2)-\mathcal{H},
		\end{aligned}
	\end{equation} 
	where $\mathcal{H}=\sum_{t,u\in\Field_{q}}^{}\la_{0}(-bt+au)\sum_{\epsilon\in \mathcal{Z}_{1}}^{} \widehat{\la_{0}}(t\epsilon+u\epsilon^2)$. Clearly, $|\mathcal{H}|\leq |\mathcal{Z}_{1}|q^2$. Moreover, we write 
	\begin{equation}\nonumber
		\begin{aligned}
			\boldsymbol{\chi}_{f,a,b}(1,1,1,1)=T_{0}+T_{1}+T_{2}+T_{3}-\mathcal{H},
		\end{aligned}
	\end{equation}
	where $T_{0}$ corresponds to the term $t=u=0$, $T_{1}$ to the term $t=0,u\neq 0$, $T_{2}$ to the term $t\neq 0, u=0$ and $T_{3}$ to the term $t,u\neq 0$. Since $q$ is odd prime power, we use Lemma \ref{L2.6} to estimate $T_i'$ s. In particular:
	
	$T_{1}=\sum_{u\in\Field_{q}^*}^{}\la_{0}(au)\sum_{\epsilon\in\Field_{q^{n}}}^{} \widehat{\la_{0}}(u\epsilon^2)\implies|T_{1}|\leq (q-1)q^{n/2}$.
	
	$T_{2}=\sum_{t\in\Field_{q}^*}^{}\la_{0}(-bt)\sum_{\epsilon\in\Field_{q^{n}}}^{} \widehat{\la_{0}}(t\epsilon)\implies T_{2}=0$.
	
	$T_{3}=\sum_{t,u\in\Field_{q}^*}^{}\la_{0}(-bt+au)\sum_{\epsilon\in\Field_{q^{n}}}^{} \widehat{\la_{0}}(t\epsilon+u\epsilon^2)\implies |T_{3}|\leq (q-1)^2q^{n/2}$, \\
	while clearly, $T_{0}=q^n$.\\
	Combining all the above, we obtain that
	\begin{equation}\lb{E100}
		|\boldsymbol{\chi}_{f,a,b}(1,1,1,1)-q^n|\leq q(q-1)q^{n/2}+ |\mathcal{Z}_{1}|q^2.
	\end{equation}
	
	Finally, assume that $R_{2}(x)\neq h(x)^{p}-h(x)+c$ for any $h(x)\in\Field_{q^{n}}(x)$ and $c\in\Field_{q^n}$. Then, by using Lemma \ref{L2.2}, we get that 
	$$|\boldsymbol{\chi}_{f,a,b}(d_{1},d_{2},h_{1},h_{2}) |\leq 2mq^{\frac{n}{2}+2},$$

	%Here we arrive at the following possibilities:\\
	%\textbf{Case 1:} If $m_{2}+2\geq m_{1}$, then as mentioned in Lemma \ref{L2.2} we get $D_{2}=2$, and
	%\begin{equation}\nonumber
	%\begin{aligned}
	%|\boldsymbol{\chi}_{f,a,b}(d_{1},d_{2},h_{1},h_{2})|\leq (2m+2)q^{\frac{n}{2}+2}.\end{aligned}
%\end{equation} 
%\textbf{Case 2:} If $m_{2}+2<m_{1}$, then we get $D_{2}=m_{1}-m_{2}$ and
%\begin{equation}\nonumber
%\begin{aligned}
%|\boldsymbol{\chi}_{f,a,b}(d_{1},d_{2},h_{1},h_{2}) |\leq 2mq^{\frac{n}{2}+2}.
%\end{aligned}
%\end{equation}  

Thus, if $(d_{1},d_{1},h_{1},h_{2})\neq (1,1,1,1)$ then based on the above discussions we get that
$$|\boldsymbol{\chi}_{f,a,b}(d_{1},d_{2},h_{1},h_{2}) |\leq 2mq^{\frac{n}{2}+2},$$
where the third and fourth ``$1$'' represent the identity in $\Field_{q}[x]$. 
Therefore, from the above discussions together with (\ref{E4}), it follows that
\begin{equation}\lb{E5}
	\begin{aligned}
		&\mathfrak{C}_{f,a,b}(l_{1},l_{2},L_{g_1},g_{2})\\&\geq \mathcal{G}\{q^n-|\mathcal{Z}_{1}|q^2-q(q-1)q^{n/2}- 2mq^{\frac{n}{2}+2}(W(l_{1})W(l_{2})W(L_{g_1})W(g_{2})-1)\}\\
		&>\mathcal{G}\{q^n- 2mq^{\frac{n}{2}+2}W(l_{1})W(l_{2})W(L_{g_1})W(g_{2})\}.
	\end{aligned}
\end{equation}
Hence $\mathfrak{C}_{f,a,b}(l_{1},l_{2},L_{g_1},g_{2})>0$, if we have $q^{\frac{n}{2}-2}> 2mW(l_{1})W(l_{2})W(L_{g_1})W(g_{2})$. Consequently, we have $(q,n)\in\mathcal{D}_{m}$ by letting $l_{1}=l_{2}=q^n-1$ and $g_{1}=g_{2}=x^n-1$, that is, provided
\begin{equation}\nonumber
	q^{\frac{n}{2}-2}>
	2m{W(q^n-1)}^2W(L)W(x^n-1).
\end{equation}
\ep
\section{Sieve Results}\lb{S4}
Here, we state some results, their proofs have been omitted as they follow on the lines of the
results in \cite{ARS} and have been used frequently in \cite{ARS, HR, SHR2021,   RAA}.
\begin{lem}\lb{L4.1}
	 Let $k,P$ be positive integers and $g,g',G,G'\in\Field_{q}[x]$ be polynomials over $\Field_{q}$  such that gcd$(k,P)=1$ and gcd$(g,G)=$gcd$(g',G')=1$. Let $\{p_1,p_2,\ldots,p_r\}$ be the set of prime factors of P, $\{g_1,g_2,\ldots,g_s\}$ be the set of irreducible factors of $G$ and $\{g_1',g_2',\ldots,g_t'\}$ be the set of irreducible factors of $G'$. Then
	\begin{equation}\lb{E6}
		\begin{aligned}
			\mathfrak{C}_{f,a,b}&(kP,kP,gG,g'G')\geq \sum_{i=1}^{r}\mathfrak{C}_{f,a,b}(p_{i}k,k,g,g')+\sum_{i=1}^{r}	\mathfrak{C}_{f,a,b}(k,p_{i}k,g,g')\\&+\sum_{j=1}^{s}\mathfrak{C}_{f,a,b}(k,k,g_{j}g,g')
			+\sum_{j=1}^{t}\mathfrak{C}_{f,a,b}(k,k,g,g_{j}g')-(2r+s+t-1)\mathfrak{C}_{f,a,b}(k,k,g,g').\\
		\end{aligned}
	\end{equation}
\end{lem} 
\begin{lem}\lb{L4.2}
	Let $l,n,q\in\mathbb{N}$, $g,g'\in\Field_{q}[x]$ be such that $q$ is an odd prime power, $n\geq 5$, $l|q^n-1$, $g'|x^n-1$ and $g|L$. Let $P$ be a prime number such that $P|q^n-1$ as well as $P\nmid l$, let $I'$ be an irreducible polynomial such that $I'|x^n-1$ as well as $I'\nmid g'$, and let $I$ be an irreducible polynomial such that $I|x^n-1$ as well as $I\nmid g$. Then we get the following bounds:
	\begin{equation}\nonumber
		\begin{aligned}
			|\mathfrak{C}_{f,a,b}(Pl,l,g,g')-\theta(P)\mathfrak{C}_{f,a,b}(l,l,g,g')|&\leq 2m\theta(P){\theta(l)}^2\Theta(g)\Theta(g'){W(l)}^2W(g)W(g')q^{n/2},\\
			|\mathfrak{C}_{f,a,b}(l,Pl,g,g')-\theta(P)\mathfrak{C}_{f,a,b}(l,l,g,g')|&\leq 2m\theta(P){\theta(l)}^2\Theta(g)\Theta(g'){W(l)}^2W(g)W(g')q^{n/2},\\
			|\mathfrak{C}_{f,a,b}(l,l,Ig,g')-\Theta(I)\mathfrak{C}_{f,a,b}(l,l,g,g')|&\leq 2m\Theta(I){\theta(l)}^2\Theta(g)\Theta(g'){W(l)}^2W(g)W(g')q^{n/2},\\
			|\mathfrak{C}_{f,a,b}(l,l,g,I'g)-\Theta(I')\mathfrak{C}_{f,a,b}(l,l,g,g)|&\leq 2m\Theta(I'){\theta(l)}^2\Theta(g)\Theta(g'){W(l)}^2W(g)W(g')q^{n/2}.\\
		\end{aligned}
	\end{equation} 
\end{lem}
%\bp
%From the definition, we have
%\begin{equation}\nonumber
%	\begin{aligned}
%		\mathfrak{C}_{f,a,b}(Pl,l,g,g)-\theta(P)&\mathfrak{C}_{f,a,b}(l,l,g,g)\\&=\mathcal{G}\underset{\underset{h_{1}|g,h_{2}|g}{P|d_{1}|lP, d_{2}|l}}{\sum}\frac{\mu}{\phi}(d_{1},d_{2},h_{1},h_{2})\underset{\underset{\la_{h_{1},\la_{h_{2}}}}{\chi_{d_{1}},\chi_{d_{2}}}}{\sum}\boldsymbol{\chi}_{f,a,b}(d_{1},d_{2},h_{1},h_{2}) . 
%	\end{aligned}
%\end{equation}
%Now, using $|\boldsymbol{\chi}_{f,a,b}(d_{1},d_{2},h_{1},h_{2})|\leq 2mq^{\frac{n}{2}+2}$, we get
%\begin{multline*}
%	|\mathfrak{C}_{f,a,b}(Pl,l,g,g)-\theta(P)\mathfrak{C}_{f,a,b}(l,l,g,g)|\\ \leq\frac{\theta(P)\theta(l)^2\Theta(g)^2}{q^2}2mq^{\frac{n}{2}+2}W(l){W(g)}^2(W(Pl)-W(l)). 
%\end{multline*}
%Clearly $W(Pl)=W(P)W(l)=2W(l)$, and thus we have 
%\begin{equation}\nonumber
%	\begin{aligned}
%		|\mathfrak{C}_{f,a,b}(Pl,l,g,g)-\theta(P)\mathfrak{C}_{f,a,b}(l,l,g,g)|\leq 2m\theta(P)\theta(l)^2\Theta(g)^2{W(l)}^2{W(g)}^2q^{{n}/{2}}. 
%	\end{aligned}
%\end{equation}
%The other bounds can also be derived in a similar manner.
%\ep
\begin{thm}[Prime Sieve]\lb{T4.3}
	Let $l,n,q\in\mathbb{N}$, $g\in\Field_{q}[x]$ be such that $q$ is an odd prime power, $n\geq 5$, $l|q^n-1$, $g|L$, and $g'|x^n-1$. Let $p_{1},p_{2},\ldots,p_{r}$ be the distinct primes dividing $q^n-1$ but not $l$, let $g_{1},g_{2},\ldots,g_{s}$ be the distinct irreducible factors of $L$ but not $g$, and let $g_{1}',g_{2}',\ldots,g_{t}'$ be the distinct irreducible factors of $(x^n-1)$ but not $g'$. Abbreviate $\mathfrak{C}_{f,a,b}(q^n-1,q^n-1,L,x^n-1)$ to $\mathfrak{C}_{f,a,b}$. Let us define
	\begin{equation}\nonumber
		\mathcal{S}:=1-2\sum_{i=1}^{r}\frac{1}{p_{i}}-\sum_{j=1}^{s}\frac{1}{q^{deg(g_{j})}}-\sum_{j=1}^{t}\frac{1}{q^{deg(g_{j}')}}, \mathcal{S}>0
	\end{equation} 
	and 
	\begin{equation}\nonumber
		\mathcal{M}:=\frac{2r+s+t-1}{\mathcal{S}}+2.
	\end{equation}
	Then $\mathfrak{C}_{f,a,b}>0$ if we have 
	\begin{equation}\lb{E7}
		q^{\frac{n}{2}-2}>
		2m{W(l)}^2W(g)W(g')\mathcal{M}.
	\end{equation}
\end{thm}
 We now deduce a modified sieving version, which has more impact compared to Theorem \ref{T4.3}, that was initially proposed by Cohen and Gupta in \cite{SA}. To progress further we require some notations and conventions which we adopt from \cite{SA}.  For $m$, a positive integer (or a polynomial over $\Field_{q}$), $\mathrm{Rad}(m)$ represents the product of distinct prime divisors (irreducible factors) of $m$. We write $\mathrm{Rad}(q^n-1)$ as the product form as $kPT$. In this representation, $k$ denotes the product of some smallest prime divisors of $q^n-1$, $T$ denotes the product of some largest prime divisors of $q^n-1$, and expressed as $T = t_{1}t_{2}\ldots t_{v}$. The remaining prime divisors are multiplied together to form $P$, represented as $P = p_{1}p_{2}\ldots p_{r}$. In a similar manner, we write $\mathrm{Rad}(x^n-1)=g'G'H'$ and $\mathrm{Rad}(L)=gGH$, where $g$ (or $g'$) is the product of certain irreducible factors of $L$ (or $x^n-1$), specifically those of minimal degree. Those irreducible factors with higher degrees are multiplied in order to the formation of $H$ (or $H'$), where $H$ (or $H'$) is expressed as the product $h_{1}h_{2}\ldots h_{w}$ (or $h_{1}'h_{2}'\ldots h_{u}'$), while the remaining factors form $G$ (or $G'$), denoted as $g_{1}g_{2}\ldots g_{s}$ (or $g_{1}'g_{2}'\ldots g_{t}'$)
\begin{thm}[Modified Prime Sieve]\lb{T5.1}
	Let $p,q,k,n,m\in\mathbb{N}$ be such that $p$ is an odd prime, $q=p^k$ and $n\geq5$. Following the above notations, let $\mathrm{Rad}(q^n-1)=kPT$, $\mathrm{Rad}(L)=gGH$, and $\mathrm{Rad}(x^n-1)=g'G'H'$. Also, let $\mathcal{S}=1-2\sum_{i=1}^{r}\frac{1}{p_{i}}-\sum_{j=1}^{s}\frac{1}{q^{deg(g_{j})}}-\sum_{j=1}^{t}\frac{1}{q^{deg(g_{j}')}}$, $\gamma_{1}=\sum_{i=1}^{v}\frac{1}{l_{i}}$, $\gamma_{2}=\sum_{j=1}^{u}\frac{1}{q^{deg(h_j')}}$,  $\gamma_{3}=\sum_{j=1}^{w}\frac{1}{q^{deg(h_{j})}}$ and $\mathcal{S}{\theta(k)}^2{\Theta(g)}\Theta(g')-(2\gamma_{1}+\gamma_{2}+\gamma_{3})>0$. Then
	\begin{equation}\nonumber
		\begin{aligned}
			q^{n/2-2}&>2m\Bigg[\Bigg({\theta(k)}^2\Theta(g)\Theta(g'){W(k)}^2W(g)W(g')(2r+s+t+2\mathcal{S}-1)+\{((v-\gamma_{1}\\&+2u-2\gamma_{2})/2)+(3v-\gamma_{1}+w+u)/2m\}+\{(m+1)(v+\gamma_{1}+w+\gamma_{2})+(u-\gamma_{2})\}\\&/(2mq^{n/2})\Bigg)\Bigg/(\mathcal{S}{\theta(k)}^2\Theta(g)\theta(g')-(2\gamma_{1}+\gamma_{2}+\gamma_{3}))\Bigg]
		\end{aligned}
	\end{equation}
	implies that $(q,n)\in\mathcal{D}_{m}$.
\end{thm}
\bp
Clearly, we have
\begin{equation}\lb{E8}
	\begin{aligned}
		\mathfrak{C}_{f,a,b}(q^n-1,&q^n-1,L,x^n-1)=\mathfrak{C}_{f,a,b}(kPT,kPT,gGH,g'G'H')\\&\geq \mathfrak{C}_{f,a,b}(kP,kP,gG,g'G')+\mathfrak{C}_{f,a,b}(L,L,H,H')-\mathfrak{C}_{f,a,b}(1,1,1,1).
	\end{aligned}
\end{equation}
By Lemma \ref{L4.1},
\begin{equation}\nonumber
	\begin{aligned}
		\mathfrak{C}_{f,a,b}(kP,kP&,gH,g'H')\\&\geq \sum_{i=1}^{r}\{\mathfrak{C}_{f,a,b}(p_{i}k,k,g,g')-\theta(p_{i})\mathfrak{C}_{f,a,b}(k,k,g,g')\}\\&+ \sum_{i=1}^{r}\{\mathfrak{C}_{f,a,b}(k,p_{i}k,g,g')-\theta(p_{i})\mathfrak{C}_{f,a,b}(k,k,g,g')\}\\&+ \sum_{j=1}^{s}\{\mathfrak{C}_{f,a,b}(k,k,g_{j}g,g')-\Theta(g_{j})\mathfrak{C}_{f,a,b}(k,k,g,g')\}\\&+ \sum_{j=1}^{t}\{\mathfrak{C}_{f,a,b}(k,k,g_{j}g,g')-\Theta(g_{j}')\mathfrak{C}_{f,a,b}(k,k,g,g')\}+\mathcal{S}\mathfrak{C}_{f,a,b}(k,k,g,g').
	\end{aligned}
\end{equation}
By using Equation $(\ref{E5})$ and Lemma \ref{L4.2}, the above inequality turns into
\begin{equation}\lb{E9}
	\begin{aligned}
		\mathfrak{C}_{f,a,b}&(kP,kP,gH,g'H')\geq\mathcal{S}{\theta(k)}^2\Theta(g)\Theta(g')[q^{n-2}-2mq^{n/2}{W(k)}^2W(g)W(g')]\\&-2m{\theta(k)}^2\Theta(g)\Theta(g'){W(k)}^2W(g)W(g')q^{n/2}\Bigg(2\sum_{i=1}^{r}\theta(p_{i})+\sum_{j=1}^{s}\Theta(g_{j})+\sum_{j=1}^{t}\Theta(g_{j}')\Bigg)\\&\geq {\theta(k)}^2\Theta(g)\Theta(g')[\mathcal{S}\cdot q^{n-2}-2mq^{n/2}{W(k)}^2W(g)W(g')(2r+s+t+2\mathcal{S}-1)].
	\end{aligned}
\end{equation}
Again, by Lemma \ref{L4.1}, we have the following
\begin{equation}\nonumber
	\begin{aligned}
		\mathfrak{C}_{f,a,b}(T,&T,H,H')-\mathfrak{C}_{f,a,b}(1,1,1,1)=\sum_{i=1}^{v}\mathfrak{C}_{f,a,b}(t_{i},1,1,1)+\sum_{i=1}^{v}\mathfrak{C}_{f,a,b}(1,t_{i},1,1)\\
		&+\sum_{j=1}^{w}\mathfrak{C}_{f,a,b}(1,1,h_{j},1)+\sum_{j=1}^{u}\mathfrak{C}_{f,a,b}(1,1,1,h_{j}')
		-(2v+w+u)\mathfrak{C}_{f,a,b}(1,1,1,1).
	\end{aligned}
\end{equation}
Therefore
\begin{equation}\lb{E10}
	\begin{aligned}
		&\mathfrak{C}_{f,a,b}(T,T,H,H')-\mathfrak{C}_{f,a,b}(1,1,1,1)\\&=\sum_{i=1}^{v}\{\mathfrak{C}_{f,a,b}(t_{i},1,1,1)-\theta(t_{i})\mathfrak{C}_{f,a,b}(1,1,1,1)\}+\sum_{i=1}^{v}\{\mathfrak{C}_{f,a,b}(1,t_{i},1,1)\\&-\theta(t_{i})\mathfrak{C}_{f,a,b}(1,1,1,1)\}+\sum_{j=1}^{w}\{\mathfrak{C}_{f,a,b}(1,1,h_{j},1)-\Theta(h_{j})\mathfrak{C}_{f,a,b}(1,1,1,1)\}\\&+\sum_{j=1}^{u}\{\mathfrak{C}_{f,a,b}(1,1,1,h_{j}')-\Theta(h_{j}')\mathfrak{C}_{f,a,b}(1,1,1,1)\}
		-(2\gamma_{1}+\gamma_{2}+\gamma_3)\mathfrak{C}_{f,a,b}(1,1,1,1).
	\end{aligned}
\end{equation}
Assume that $l$ be any prime divisor of $q^n-1$. From the definition of $\mathfrak{C}_{f,a,b}(l_{1},l_{2},g_{1},g_{2})$, it follows that
\begin{equation}\nonumber
	\begin{aligned}
		|\mathfrak{C}_{f,a,b}(l,1,1,1)-\theta(l)\mathfrak{C}_{f,a,b}(1,1,1,1)|=\frac{\theta(l)}{q^2}\Bigg|\sum_{\chi_{l}}^{}\chi_{f,a,b}(l,1,1,1)\Bigg|,
	\end{aligned}
\end{equation}
where 
\begin{equation}\nonumber
	\begin{aligned}
		|\chi_{f,a,b}(l,1,1,1)|&=\Bigg|\sum_{t,u\in\Field_{q}}^{}\la_{0}(-bt+au)\sum_{\epsilon\in{{{\Field}_{q^n}}}\smallsetminus\mathcal{Z}_{1}}\chi_{l}(\epsilon)\widehat{\la_{0}}(t\epsilon+u\epsilon^2)\Bigg|\\&\leq 2q^{\frac{n}{2}+2}+mq^2.
	\end{aligned}
\end{equation}
Hence, $|\mathfrak{C}_{f,a,b}(l,1,1,1)-\theta(l)\mathfrak{C}_{f,a,b}(1,1,1,1)|\leq \theta(l)(2q^{n/2}+m)$. Analogously,
\begin{equation}\nonumber
	\begin{aligned}
		|\chi_{f,a,b}(1,l,1,1)|&=\Bigg|\sum_{t,u\in\Field_{q}}^{}\la_{0}(-bt+au)\sum_{\epsilon\in{{{\Field}_{q^n}}}\smallsetminus\mathcal{Z}_{1}}\chi_{l}(f(\epsilon))\widehat{\la_{0}}(t\epsilon+u\epsilon^2)\Bigg|\\&\leq (m+1)q^{\frac{n}{2}+2}+q^2,
	\end{aligned}
\end{equation}
which further implies,  $|\mathfrak{C}_{f,a,b}(1,l,1,1)-\theta(l)\mathfrak{C}_{f,a,b}(1,1,1,1)|\leq \theta(l)((m+1)q^{n/2}+1)$. Again, let $h$, $h'$ be an irreducible factors of $L$ and $(x^n-1)$ respectively, where $\la_{h}(\epsilon)=\widehat{\la_{0}}(y\epsilon)$ and $\la_{h'}(\epsilon)=\widehat{\la_{0}}(y'\epsilon)$, for $y,y'\in\Field_{q^{n}}^*$. Thus, we get
\begin{equation}\nonumber
	\begin{aligned}
		|\chi_{f,a,b}(1,1,h,1)|&=\Bigg|\sum_{t,u\in\Field_{q}}^{}\la_{0}(-bt+au)\sum_{\epsilon\in{{{\Field}_{q^n}}}\smallsetminus\mathcal{Z}_{1}}\la_{h}(\epsilon)\widehat{\la_{0}}(t\epsilon+u\epsilon^2)\Bigg|\\
		&=\Bigg|\sum_{t,u\in\Field_{q}}^{}\la_{0}(-bt+au)\sum_{\epsilon\in{{{\Field}_{q^n}}}\smallsetminus\mathcal{Z}_{1}}\widehat{\la_{0}}((t+y)\epsilon+u\epsilon^2)\Bigg|\\
		&\leq q^{\frac{n}{2}+2}+(m+1)q^2.
	\end{aligned}
\end{equation}
Therefore, $|\mathfrak{C}_{f,a,b}(1,1,h,1)-\Theta(h)\mathfrak{C}_{f,a,b}(1,1,1,1)|\leq \Theta(h)(q^{n/2}+(m+1))$.
Similarly, we have
\begin{equation}\nonumber
	\begin{aligned}
		|\chi_{f,a,b}(1,1,1,h')|&=\Bigg|\sum_{t,u\in\Field_{q}}^{}\la_{0}(-bt+au)\sum_{\epsilon\in{{{\Field}_{q^n}}}\smallsetminus\mathcal{Z}_{1}}\la_{h'}(f(\epsilon))\widehat{\la_{0}}(t\epsilon+u\epsilon^2)\Bigg|\\
		&=\Bigg|\sum_{t,u\in\Field_{q}}^{}\la_{0}(-bt+au)\sum_{\epsilon\in{{{\Field}_{q^n}}}\smallsetminus\mathcal{Z}_{1}}\widehat{\la_{0}}(t\epsilon+u\epsilon^2+y'f(\epsilon))\Bigg|\\
		&\leq (2m+1)q^{\frac{n}{2}+2}+q^2,
	\end{aligned}
\end{equation}
and thus $|\mathfrak{C}_{f,a,b}(1,1,1,h)-\Theta(h')\mathfrak{C}_{f,a,b}(1,1,1,1)|\leq \Theta(h')((2m+1)q^{n/2}+1)$. Also, by (\ref{E100}), we get that $|\mathfrak{C}_{f,a,b}(1,1,1,1)|\leq q^{n-2}+q^{n/2}+ |\mathcal{Z}_{1}|$. Thus, we have
\begin{equation}\nonumber
	\begin{aligned}
		\mathfrak{C}_{f,a,b}(T,T,H,H')-&\mathfrak{C}_{f,a,b}(1,1,1,1)\\&\geq-\sum_{i=1}^{v}\theta(t_{i})(2q^{n/2}+m)-\sum_{i=1}^{v}\theta(t_{i})((m+1)q^{n/2}+1)\\&-\sum_{j=1}^{w}\Theta(h_{j})(q^{n/2}+(m+1))-\sum_{j=1}^{u}\Theta(h_{j}')((2m+1)q^{n/2}+1)\\&
		-(2\gamma_{1}+\gamma_{2}+\gamma_{3})\{q^{n-2}+q^{n/2}+ (m+1)\}.
	\end{aligned}
\end{equation}
Note that $\sum_{i=1}^{v}\theta(t_{i})=v-\gamma_{1}$, $\sum_{j=1}^{w}\Theta(h_{j})=w-\gamma_{3}$, and $\sum_{j=1}^{u}\Theta(h_{j}')=u-\gamma_{2}$. Hence, we get the following
\begin{equation}\nonumber
	\begin{aligned}
		\mathfrak{C}_{f,a,b}(T,T,H,H')-\mathfrak{C}_{f,a,b}(1,&1,1,1)\geq-(v-\gamma_{1})(2q^{n/2}+m)-(v-\gamma_{1})\{(m+1)q^{n/2}+1\}\\&-(w-\gamma_{3})\{q^{n/2}+(m+1)\}-(u-\gamma_{2})\{(2m+1)q^{n/2}+1\}\\&-(2\gamma_{1}+\gamma_{2}+\gamma_{3})\{q^{n-2}+q^{n/2}+ (m+1)\},
	\end{aligned}
\end{equation}
which further implies that
\begin{equation}\lb{E11}
	\begin{aligned}
		\mathfrak{C}_{f,a,b}(T,T,&H,H')-\mathfrak{C}_{f,a,b}(1,1,1,1)\geq-q^{n/2}\{(m(v-\gamma_{1}+2u-2\gamma_{2}))+(3v-\gamma_{1}+w+u)\}\\&-\{(m+1)(v+\gamma_{1}+w+\gamma_{2})+(u-\gamma_{2})\}-(2\gamma_{1}+\gamma_{2}+\gamma_{3})q^{n-2}.\\&
	\end{aligned}
\end{equation}
Substituting (\ref{E9}) and (\ref{E11}) in (\ref{E8}) we get
\begin{equation}\nonumber
	\begin{aligned}
		&\mathfrak{C}_{f,a,b}(q^n-1,q^n-1,L,x^n-1)\\&\geq {\theta(k)}^2\Theta(g)\Theta(g')[\mathcal{S}\cdot q^{n-2}-2mq^{n/2}{W(k)}^2W(g)W(g')(2r+s+t+2\mathcal{S}-1)]\\&
		-q^{n/2}\{(m(v-\gamma_{1}+2u-2\gamma_{2}))+(3v-\gamma_{1}+w+u)\}-\{(m+1)(v+\gamma_{1}+w+\gamma_{2})\\&+(u-\gamma_{2})\}-(2\gamma_{1}+\gamma_{2}+\gamma_{3})q^{n-2},
	\end{aligned}
\end{equation}
that is,
\begin{equation}\nonumber
	\begin{aligned}
		\mathfrak{C}_{f,a,b}\geq& \Bigg[({\mathcal{S}\theta(k)}^2{\Theta(g)\Theta(g')}-(2\gamma_{1}+\gamma_{2}+\gamma_{3}))q^{n-2}-2mq^{n/2}\Bigg\{{\theta(k)}^2\Theta(g)\Theta(g'){W(k)}^2W(g) \\&\times W(g')(2r+s+t+2\mathcal{S}-1)+\{((v-\gamma_{1}+2u-2\gamma_{2})/2)+(3v-\gamma_{1}+w+u)/2m\}\\&+\{(m+1)(v+\gamma_{1}+w+\gamma_{2})+(u-\gamma_{2})\}/(2mq^{n/2})\Bigg\}\Bigg].
	\end{aligned}
\end{equation}
Thus 
\begin{equation}\nonumber
	\begin{aligned}
		q^{n/2-2}&>2m\Bigg[\Bigg({\theta(k)}^2\Theta(g)\Theta(g'){W(k)}^2W(g)W(g')(2r+s+t+2\mathcal{S}-1)+\{((v-\gamma_{1}\\&+2u-2\gamma_{2})/2)+(3v-\gamma_{1}+w+u)/2m\}+\{(m+1)(v+\gamma_{1}+w+\gamma_{2})+(u-\gamma_{2})\}\\&/(2mq^{n/2})\Bigg)\Bigg/(\mathcal{S}{\theta(k)}^2\Theta(g)\theta(g')-(2\gamma_{1}+\gamma_{2}+\gamma_{3}))\Bigg]
	\end{aligned}
\end{equation}implies $\mathfrak{C}_{f,a,b}>0$, that is, $(q,n)\in\mathcal{D}_{m}$.

Notice that Theorem \ref{T4.3} is a particular case of Theorem \ref{T5.1}, which can be deduced by letting $v=w=u=\gamma_{1}=\gamma_{2}=\gamma_3=0$.  
\ep

\section{Numerical Example}
In this section we utilize our results to find out the presence of elements having desired properties. The results that are mentioned earlier apply to the arbitrary finite field $\Field_{q^n}$ of arbitrary odd characteristic. For illustration, we explicitly determine each pair $(q,n)$ belonging to $\mathcal{D}_2$, where $q=7^k$ and $n\geq 6$. Here let us split our calculations into two parts. Initially, we identify the exceptions $(q,n)$ for $n\geq 8$, and subsequently, we execute the possible exceptions for $n=6,7$. In this article, SageMath \cite{Sm} serves as the computational tool for all significant calculations. From Theorem \ref{T4.3}, it follows that $(q,n)\in\mathcal{D}_2$ if we have
\begin{equation}\lb{E12}
	q^{\frac{n}{2}-2}>4 ~{W(l)}^2W(g)W(g')\mathcal{M}.
\end{equation}
Also, by Theorem \ref{T3.1}, $(q,n)\in\mathcal{D}_2$ if we have
\begin{equation}\lb{E13}
	q^{\frac{n}{2}-2}>
	4{W(q^n-1)}^2W(L)W(x^n-1).
\end{equation}
\textbf{Part I:} Rewrite $n$ as $n = n'\cdot q^i$; $i\geq 0$, where $q,n'\in\mathbb{N}$ be such that $q$ is a prime power, being co prime to $n'$. Furthermore, assume that $e$ be the order of $q$ modulo $n'$, where gcd$(n',q)=1$. Following [{\cite{RH}}, Theorems $2.45$ and $2.47$], $x^{n'}-1$ can be factorized into the product of irreducible polynomials over $\Field_{q}$ in such a way that degree of each factor must be less than or equal to $e$. Let $N_0$ be the number of the irreducible factors of $x^{n'}-1$ degree less than $e$.

Denote $I_{n'}$ as the cardinality of the set containing the irreducible factors of $x^{n'}-1$ over $\Field_{q}$ such that degree of each factor is less than $e$, and let the ratio $\frac{N_{0}}{n'}$ be denoted by $\pi(q,n')$. Observe that the set containing the irreducible factors of $x^{n}-1$ over $\Field_{q}$ and the set containing irreducible factors of $x^{n'}-1$ over $\Field_{q}$ are of equal cardinality, and this results in $n\pi(q,n)=n'\pi(q,n')$. For further computations, we shall use bounds for $\pi(q,n)$, that is provided in the following lemma.
\begin{lem}[\cite{SS}, Lemma 6.1, Lemma 7.1]\lb{L6.1}
	Let $q=7^k$ and $n'>4$ be such that $7\nmid n'$. Let $n_{1}'=gcd(n',q-1)$. Then the following hold:
	\begin{itemize}
		\item[(i)] When $n'=2n_{1}'$, then we have $e=2$ and $\pi(q,n')=1/2$.
		\item[(ii)]When $n'=4n_{1}'$ and $q\equiv 1(mod~ 4)$, then we have $e=4$ and  $\pi(q,n')=3/8$.
		\item[(iii)] When $n'=6n_{1}'$ and $q\equiv 1(mod~ 6)$, then we have $e=6$ and  $\pi(q,n')=13/36$.
		\item[(iv)] Otherwise, we have $\pi(q,n')\leq 1/3$.
	\end{itemize}
\end{lem}
\begin{lem}\lb{L6.2}
	Suppose that $q=p^k$; $k\in\mathbb{N}$ and $n=n'\cdot q^{i}$; $i\in \mathbb{N}\cup \{0\}$, where gcd($n',q)=1$ in addition with $n'\nmid q-1$. Assume that $e(>2)$ be the order of $q$ (mod $n'$). Further, let $l=q^n-1$ and $g=g'/(x-1)$, where $g'$ is the product of all irreducible factors of $x^{n'}-1$ along with each one have degree less than $e$. Then, following Theorem \ref{T4.3}, we get that $\mathcal{M} <2n'$. 
\end{lem}
\begin{proof}
	The proof is omitted here, as it can be derived from [\cite{MASI}, Lemma 10].
\end{proof}
Here we determine the pairs $(q,n)\in\mathcal{D}_4$, when $q=7^k$ and $n\geq 8$. Henceforth, let us suppose that $n\geq 8$ and $n=n'\cdot 7^i$, where $\gcd(7,n')=1$. Then we have $W(x^n-1)=W(x^{n'}-1)$. Also, note that $W(L)=\frac{W(x^n-1)}{2}$.
\begin{lem}\lb{L6.3}
	Let $q=7$ and  $n=n'\cdot 7^i$, where $\gcd(7,n')=1$. Then $(7,n)\in\mathcal{D}_2$ for all $n\geq 8$ except for $n=8,9,10,12,$ and $18$.
\end{lem}
\bp
Firstly, assume that $n'\nmid q^2-1$. Then we must have $n'\geq 5$ and by Lemma \ref{L6.1}, we have $\pi(7,n')\leq 1/3$ unless $n'=36$, since then $n'=6n_{1}'$ and $\pi(q,n')=13/36$. Let $l=7^n-1$ and $g=g'/(x-1)$, where $g'$ is the product of all irreducible factors of $x^{n'}-1$ of degree less than $e$. Thus, it follows from Lemmas \ref{L2.3}, \ref{L6.2} and Inequality (\ref{E13}) that, $(7,n)\in\mathcal{D}_2$ if we have
\begin{equation}\nonumber
	7^{n/2-2}>4~{\mathcal{C}}^2~7^{2n/r}~2^{2n/3-1}~2n.
\end{equation}
Observe that the above inequality holds for $r=9.7$ and $n\geq 434$. For $n\leq 433$, we test Inequality (\ref{E13}), and get that $(7,n)\in\mathcal{D}_2$ except for $n= 9, 10, 11, 15, 18, 19, 20, 27,$\\$ 30, 32$. However, we apply a sieve approach to handle these exceptions and get the values of $l$, $g$, $g'$, $\mathcal{S}$, and $\mathcal{M}$ for all pairs $(q,n)$, as listed in Table \ref{Table3}. Thus we get $(7,n)\in\mathcal{D}_2$ unless $n=9,10$ and $18$. We now consider $n'=36$ and by testing the inequality $7^{36\cdot 7^i/2-2}>288\cdot\mathcal{C}^2\cdot7^{72\cdot 7^i/r}\cdot 2^{25}$, we get $(7,36\cdot 7^i)\in\mathcal{D}_2$ for $i\geq 1$ and $r=10$. For the sole remaining pair $(7,36)$, we verify Inequality (\ref{E12}) and find the values of $l$, $g$, $g'$, $\mathcal{S}$, and $\mathcal{M}$ (listed in Table \ref{Table3}).

Secondly, assume that $n'|q^2-1$. By Lemmas \ref{L2.3}, \ref{L2.4} and Inequality (\ref{E13}), $(7,n)\in\mathcal{D}_2$ if we have
\begin{equation}\nonumber
	7^{n'\cdot 7^i/2-2}>4~{\mathcal{C}}^2~7^{2n'\cdot 7^i/r}~2^{2n'-1}.
\end{equation}
Taking $r=9$, the above inequality holds for $i\geq 3$, when $n'=1$, for $i\geq 2$, when $n'=2,3,4,6,8,12$ and for $i\geq 1$, when $n'=16,24,48$. Hence, for $n\geq 8$, $(7,n)\in\mathcal{D}_2$ except when $n=8,12,14,16,21,24,28,42,48,49,56,84$. For this exceptions, by testing Inequality (\ref{E13}), we get that $(7,n)\in\mathcal{D}_2$ except when $n=8,12,14,16,24,48$. However, for the remaining pairs, we verify Inequality (\ref{E12}) and find the values of $l$, $g$, $g'$, $\mathcal{S}$ and $\mathcal{M}$ (listed in Table \ref{Table3}) and we get that $(7,n)\in\mathcal{D}_2$ unless $n=8, 12$.
\ep
\begin{lem}
	Let $q=49$ and  $n=n'\cdot 7^i$, where $\gcd(7,n')=1$. Then $(49,n)\in\mathcal{D}_2$ for all $n\geq 8$.
\end{lem}
\bp
Firstly, assume that $n'\nmid q^2-1$. Then $n'\nmid q-1$ and thus by Lemma \ref{L2.4}, we have $W(x^{n'}-1)\leq 2^{\frac{3}{4}n'}$. Then, Inequality (\ref{E13}) holds true if $q^{n/2-2}>4~ \mathcal{C}^2 q^{2n/r}~2^{3n/2-1}$.  By choosing $r=10.5$, the latter inequality holds for all $n\geq 372$. For $n\leq 371$, we verify Inequality (\ref{E13}) and get that $(49,n)\in\mathcal{D}_2$ unless $n=9, 18$. However, we apply a sieve approach to handle these exceptions and get the values of $l$, $g'$ , $g$, $\mathcal{S}$, and $\mathcal{M}$ for all pairs $(q,n)$, as listed in Table \ref{Table3}. Thus we get $(49,n)\in\mathcal{D}_2$ for all $n$.

Secondly, assume that $n'|q^2-1$.  By Lemmas \ref{L2.3}, \ref{L2.4} and the Inequality (\ref{E13}), $(49,n)\in\mathcal{D}_2$ if we have
\begin{equation}\nonumber
	49^{n'\cdot 7^i/2-2}>4~{\mathcal{C}}^2~49^{2n'\cdot 7^i/r}~2^{2n'-1}.
\end{equation}Taking $r=9$, the above inequality holds for $i\geq 2$, when $n'=1, 2, 3, 4$  and for $i\geq 1$, otherwise. Thus for $n\geq 8$, $(49,n)\in\mathcal{D}_2$ except when $n'=$ 8, 10, 12, 14, 15, 16, 20, 21, 24, 25, 28, 30, 32, 40, 48, 50, 60, 75, 80, 96, 100, 120, 150, 160, 200, 240, 300, 400, 480, 600, 800, 1200, 2400. For each of the values of $n$, we verify Inequality (\ref{E13}) and get that $(49,n)\in\mathcal{D}_2$ unless $n=8, 10, 12, 15, 16, 20, 24, 48$. However, we apply a sieve approach to handle these exceptions and get the values of $l$, $g$, $\mathcal{S}$, and $\mathcal{M}$ for all pairs $(q,n)$, as listed in Table \ref{Table3}. Hence, we get $(49,n)\in\mathcal{D}_2$ for all $n$.
\ep
\begin{lem}
	Let $k\in\mathbb{N}$ and $q=7^k$. Then $(q,n)\in\mathcal{D}_2$ for $n\geq 8$ and $k\geq 3$.
\end{lem}
\bp
From Lemma \ref{L2.3}, $W(q^n-1)<C\cdot (q^n-1)^{1/r}$ for some real number $r>0$ and also we have $W(x^n-1)\leq 2^n$. Therefore, by using Inequality (\ref{E13}), $(q,n)\in\mathcal{D}_2$ if
\begin{equation}\lb{E14}
	q^{n/2-2}>4\cdot \mathcal{C}^2\cdot q^{2n/r}\cdot 2^{2n-1}.
\end{equation}
\begin{center}
	\begin{table}[h!]
		\centering
		\caption{Values of $k$ such that for $n\geq n_k$, (\ref{E14}) is true.}
		\begin{tabular}{|c|c|c|}
			\hline $r$ & $k$ & $n_k$ \\
			\hline 10 & \{3\} & 149 \\
			9.0 & \{4\} & 56\\
			8.5 & \{5\} & 35\\
			8.5 & \{6\} & 27 \\
			8.5 & \{7\} & 23 \\
			8.5 & \{8\} & 20 \\
			8.5 & \{9\} & 18 \\
			9 & \{10\} & 17 \\
			9 & \{11\} & 16 \\
			9 & \{12\} & 15 \\
			9 & \{13,14\} & 14 \\
			9 & \{15,16\} & 13 \\
			9 & \{17,19,20,20\} & 12 \\
			9 & \{21,23,\ldots,26\} & 11 \\
			9 & \{27,29,\ldots,39\} & 10 \\
			9.5 & \{40,42,\ldots,72\} & 9 \\
			\hline
		\end{tabular}
		\label{Table1}
	\end{table}
\end{center}
For $r=9.5$, Lemma \ref{L2.3} gives $\mathcal{C}<1.46\times 10^7$ and thus Inequality (\ref{E14}) holds for $n\geq 8$ and $k\geq 73$. For each $3\leq k\leq 72$ and for proper choice of `$r$', we get $n_{k}$'s such that for all $n\geq n_{k}$, Inequality (\ref{E14}) is satisfied, that are listed in Table \ref{Table1}.  
We calculate the values of $W(q^n-1)$ and $W(x^n-1)$ precisely for each of the above values of $k$ and $n$ and verify the inequality
\begin{equation}\nonumber
	q^{\frac{n}{2}-2}>
	2\cdot {W(q^n-1)}^2{W(x^n-1)}^2.
\end{equation}	
Consequently, we get $(q,n)\in\mathcal{D}_2$ unless $(q,n)$ equals  $(7^3,8),(7^3,9),(7^3,10), (7^3,12),  $\\$ (7^3,18), (7^4,8),(7^4,9) , (7^4,10), (7^4,12), (7^4,15), (7^5,8), (7^6,8), (7^6,9)$. By applying a sieve approach to handle these exceptions, we get the values of $l$, $g'$, $g$, $\mathcal{S}$, and $\mathcal{M}$ for all pairs $(q,n)$, as listed in Table \ref{Table3}. Thus $(7^k,n)\in\mathcal{D}_2$ for all $n\geq 8$ and $k\geq 3$.
\ep
\begin{center}
	\begin{table}[h!]
		\centering
		\caption{Choices of $l$, $g$, and $g'$ to satisfy Theorem \ref{T4.3}.}
		\begin{tabular}{|c|c|c|c|c|c|c|}
			\hline Sr. No.& $(q,n)$ & $l$  & $g$&$g'$\\
			
			\hline
	$1$&$(7,11)$&$2$&$1$&$1$\\
	
	$2$&$(7,14)$&$6$&$1$&$1$\\

$3$&$(7,15)$&$6$&$1$&$1$\\
	
	$4$&$(7,16)$&$6$&$x^2+x+6$&$x^2+x+6$\\

$5$&	$(7,19)$&$6$&$1$&$1$\\
	
$6$	&$(7,20)$&$6$&$x+1$&$x+1$\\

$7$&$(7,27)$&$6$&$x+3$&$x+3$\\

$8$&$(7,30)$&$6$&$x^3+x^2+5x+5$&$x^4+2x^2+4$\\

$9$&$(7,32)$&$6$&$x^2+1$&$x^2+6$\\

	$10$&$(7,36)$&$30$&$\frac{x^6+6}{x+6}$&$x^6+6$\\

$11$	&$(7,48)$&$30$&$\frac{x^{24}+6}{x+6}$&$x^{24}+6$\\

	\hline	
			%$(7,15)$&$6$&$x+3$&$0.255566642087676$&$68.5188534040679$\\
			%$(7,16)$&$6$&$x^2+6$&$0.194961580806272$&$109.713529574153$\\

			%$(7,24)$&$30$&$x^6+6$&$0.271667188760882$&$123.472159190509$\\

			%$(7,40)$&$6$&$x^2+6$&$0.219823120479101$&$179.415368842914$\\
		
			%\hline
$12$&$(7^2,8)$&$1$&$1$&$1$\\			
			
$13$&$(7^2,9)$&$6$&$x+3$&$x+3$\\

$14$&	$(7^2,10)$&$6$&$x+1$&$x+1$\\

$15$&	$(7^2,12)$&$30$&$x+1$&$x+1$\\

$16$&	$(7^2,15)$&$6$&$x+3$&$x+3$\\

$17$&	$(7^2,16)$&$30$&$x+1$&$x+1$\\

$18$	&$(7^2,18)$&$6$&$x+1$&$x+1$\\
			
$19$	&$(7^2,20)$&$6$&$x+1$&$x+1$\\
			%$(7^2,11)$&$2$&$x+6$&$0.244588852161447 $&$55.1504191017628$\\
		
$20$&	
$(7^2,24)$&$5416710$&$x+1$&$x+1$\\		
		
$21$&	$(7^2,48)$&$38355723510$&$x^{36}+x^{24}+x^{12}+1$&$x^{36}+x^{24}+x^{12}+1 $\\		
		
			%$(7^2,30)$&$330$&$x+1$&$0.353135978712364$&$169.074451646448$\\
		
			\hline 
		$22$&	$(7^3,8)$&$6$&$1$&$1$\\
		$23$&	$(7^3,9)$&$2$&$x+3$&$x+3$\\
		$24$&	$(7^3,10)$&$6$&$x+1$&$x+1$\\
		$25$&	$(7^3,12)$&$6$&$x+1$&$x+1$\\
		$26$&	$(7^3,18)$&$6$&$x+1$&$x+1$\\
			\hline
		$27$&	$(7^4,8)$&$6$&$x+1$&$x+1$\\
		$28$&	$(7^4,9)$&$6$&$x+3$&$1$\\
		$29$&	$(7^4,10)$&$6$&$x+1$&$1$\\
		$30$&	$(7^4,12)$&$6$&$x+1$&$x+1$\\
		$31$&	$(7^4,15)$&$30$&$x+3$&$x+1$\\
			\hline
		$32$&	$(7^5,8)$&$6$&$x+1$&$1$\\
			\hline
		$33$&	$(7^6,8)$&$6$&$x+1$&$1$\\
		$34$&	$(7^6,9)$&$2$&$x+3$&$1$\\
			\hline
		\end{tabular}
		\label{Table3}
	\end{table}
\end{center}
{\bf{Part II:}} In this part, we execute computations for $n=6,7$. The following lemma will be utilized in this part for computation.
\begin{lem}\lb{L5.6}
	Let $M\in\mathbb{N}$ such that $\om(M)\geq 2828$. Then we have $W(M)<M^{1/13}$.
\end{lem}
\begin{proof}
	Let $S=\{2,3,\ldots,25673\}$ be the set containing first $2828$ primes. Clearly the product of all elements in $S$ surpasses $2.24\times 10^{11067}$. Let us decompose $M$ as the product of two co prime positive integers $M_{1}$ and $M_{2}$ such that prime divisors of $M_{1}$ come from the least $2828$ prime divisors of $M$ and remaining prime divisors are divisors of $M_{2}$. Therefore, $M_{1}^{1/13}>2.16\times10^{851}$ , on the other hand $W(M_{1})<2.06\times 10^{851}$. Hence we draw the conclusion, as  $p^{1/13}>2$ for any prime $p>25673$. 
\end{proof}
First we assume that $\om(q^n-1)\geq 2828$. Note that $W(x^n-1)\leq 2^7$ and thus by Lemma \ref{L5.6}, $(q,n)\in\mathcal{D}_2$ if we have $q^{\frac{n}{2}-2}>2\cdot q^{\frac{2n}{13}}\cdot 2^{14}$, that is, if $q^n>2^{\frac{390n}{9n-52}}$ then $(q,n)\in\mathcal{D}_2$. But $n\geq 6$ gives $\frac{390n}{9n-52}\leq 1170$. Hence, if $q^{n}>2^{1170}$ then $(q,n)\in\mathcal{D}_2$, which is valid when $ \om(q^n-1)\geq 2828$. To proceed further, we follow \cite{HR}. Let us suppose that $88\leq \om(q^n-1)\leq 2827$. In Theorem \ref{T4.3}, choose $g=L$, $g'=x^n-1$ and $l$ is assumed to be the product of least $88$ prime divisors of $q^n-1$, that is, $W(l)=2^{88}$, then $r\leq 2739$ and $\mathcal{S}$ assumes its minimum positive value when $\{p_{1},p_{2},\ldots,p_{2739}\}=\{461,463,\ldots,25667\}$. This gives $\mathcal{S}>0.0044306$ and $\mathcal{M} <1.24\times 10^6$. Thus $4\mathcal{M} {W(l)}^2W(g)W(g')<3.8797485\times 10^{63}=R$(say). By Sieve variation, we get $(q,n)\in\mathcal{D}_2$ if we have $q^{\frac{n}{2}-2}>R$, that is, if $q^n>R^{\frac{2n}{n-4}}$. Since $n\geq 6$ implies $\frac{2n}{n-4}\leq 6$, we have $(q,n)\in\mathcal{D}_2$ if $q^n>3.410527\times 10^{381}$. Hence, $\om(q^n-1)\geq 157$ gives $(q,n)\in\mathcal{D}_2$.
\begin{center}
	\begin{table}[h!]
		\centering
		\caption{}
		\begin{tabular}{|c|c|c|c|c|}
			\hline  $a\leq \om(q^n-1)\leq b$ & $W(l)$ & $\mathcal{S}>$& $\mathcal{M}<$& $4\mathcal{M} {W(l)}^2W(g)W(g')<$\\
			\hline   $a=17, b=156$ & $2^{17}$ & $0.0238003$ &$11640.489$&$6.55302\times 10^{18}$ \\
			$a=9, b=58$ & $2^{9}$ &$0.0004319$& $224570.1187$& $1.92905\times 10^{15}$\\
			$a=9, b=50$ & $2^{9}$ &$0.0638523$& $1270.5506$& $1.09140\times10^{13}$\\
			\hline
		\end{tabular}
		\label{Table2}
	\end{table}
\end{center}
We repeat the steps given in Theorem \ref{T4.3} by using the data given in the second column of Table \ref{Table2}. Therefore we get, $(q,n)\in\mathcal{D}_2$ if we have $q^{\frac{n}{2}-2}>1.09140\times10^{13}$. This gives the scenarios that  $n=6, q>1.09140\times10^{13}$; $n=7, q>(1.09140\times10^{13})^{{2}/{3}}$. Thus, the only possible exceptions are $(7,6),(7^2,6),\ldots,(7^{15},6)$; $(7,7),(7^2,7),\ldots,(7^{10},7)$.  From Table \ref{Table4}, it follows that Theorem \ref{T4.3} holds for $(7^5,6),(7^6,6),\ldots,(7^{15},6)$;$(7^3,7),(7^,7),\ldots,(7^{10},7)$. Thus the only possible exceptions are $(7,6),(7^2,6),(7^3,6),(7^4,6)$ and $(7,7),(7^2,7)$.

Finally, for all the mentioned possible exceptions, we verified Theorem \ref{T5.1} and get that none of the pairs satisfy the modified prime sieving criterion for suitable values of $k,P, L, g, G$ and $H$. Although Theorem \ref{T5.1} is not effective in the above exceptional pairs, it may play a vital role in other finite fields whose characteristic is not $7$. Thus, compiling all the computational results of this section, we conclude the following result.

\begin{thm}
	Let $q,k,n\in\mathbb{N}$ such that $q=7^k$ and $n\geq 6$. Then $(q,n)\in\mathcal{D}_2$ except the following possibilities:
	\begin{itemize}
		\item[1.] $q=7,7^2,7^3,7^4~\text{and}~ n=6$;
		\item[2.] $q=7,7^2~\text{and}~ n=7$;
		\item[3.] $q=7~\text{and}~ n=8,9,10,12,18$.
	\end{itemize}
\end{thm}
\begin{center}
	\begin{table}[h!]
		\centering
		\caption{Choices of $l$, $g$, and $g'$ to satisfy Theorem \ref{T4.3}.}
		\begin{tabular}{|c|c|c|c|c|c|c|}
			\hline Sr. No.& $(q,n)$ & $l$  & $g$&$g'$\\
			\hline
		$1$&	$(7^5,6)$&$6$&$1$&$1$\\
		$2$&	$(7^6,6)$&$6$&$x+1$&$x+1$\\
		$3$&	$(7^7,6)$&$2$&$x+1$&$x+1$\\
		$4$&	$(7^8,6)$&$6$&$x+1$&$x+1$\\
		$5$&	$(7^9,6)$&$2$&$x+1$&$x+1$\\
		$6$&	$(7^{10},6)$&$30$&$x+1$&$x+1 $\\
		$7$&	$(7^{11},6)$&$2$&$x+1$&$x+1$\\
		$8$&	$(7^{12},6)$&$6$&$x+1$&$x+1$\\
		$9$&	$(7^{13},6)$&$2$&$x+1$&$x+1$\\
		$10$&	$(7^{14},6)$&$6$&$x+1$&$x+1$\\
		$11$&	$(7^{15},6)$&$6$&$x+1$&$x+1$\\
			%$(7^{16},6)$&$6$&$x+1$&$0.0950709747906092$&$454.293668963699$\\
			\hline
		$12$&	$(7^3,7)$&$2$&$1$&$1$\\
		$13$&	$(7^4,7)$&$6$&$1$&$1$\\
		$14$&	$(7^5,7)$&$2$&$1$&$1$\\
		$15$&	$(7^6,7)$&$2$&$1$&$1$\\
		$16$&	$(7^7,7)$&$2$&$1$&$1$\\
		$17$&	$(7^8,7)$&$6$&$1$&$1$\\
		$18$&	$(7^9,7)$&$2$&$1$&$1$\\
		$19$&	$(7^{10},7)$&$2$&$1$&$1$\\
			%$(7^{11},7)$&$2$&$x+6$&$0.262157487894487$&$66.8465170174432$\\
			\hline
		\end{tabular}
		\label{Table4}
	\end{table}
\end{center}

\section{Acknowledgments}
We sincerely appreciate and acknowledge the reviewers for their helpful comments and suggestions. First author is supported by the National Board for Higher Mathematics (NBHM), Department of Atomic Energy (DAE), Government of India, Ref No.  0203/6/2020-R\&D-II/7387.
\section{Statements and Declarations}
There are no known competing financial interests or personal relationships that could have influenced this paper's findings. All authors are equally contributed.

\end{document}